\RequirePackage{fix-cm} 

\documentclass[a4paper]{article}
\makeatletter
\def\cl@chapter{}
\makeatother

\usepackage[utf8]{inputenc}
\usepackage{amssymb}
\usepackage{amsmath}
\usepackage{amsthm}

\usepackage{graphicx}
\usepackage{enumitem}
\usepackage{amsfonts}
\usepackage{setspace}
\usepackage[usenames,dvipsnames]{xcolor}
\usepackage[pageanchor=false,colorlinks=true,linkcolor=blue,citecolor=blue,urlcolor=blue]{hyperref}
\usepackage{tikz}
\usepackage{tikz-3dplot}
\usepackage{pgfplots}
\usepackage[normalem]{ulem} 
 \usepackage{geometry} 

 \usepackage{cleveref}

\usetikzlibrary{calc,intersections,positioning,angles,quotes}
\usetikzlibrary{3d,shapes}
\pgfplotsset{compat=1.11}

\newcommand\bR{\mathbb{R}}

\newcommand\fix{\mathrm{fix}}

\newcommand{\tspan}{\mathrm{span}}

 \newcommand{\ls}{\langle}
 \newcommand{\rs}{\rangle}

\newtheorem{theorem}{Theorem}[section]

\newtheorem{definition}{Definition}[section]

\begin{document}

\title{A note on the convergence of  RED algorithms under minimal hypotheses and  open questions}
 \author{Yann  Traonmilin* and J.-F. Aujol \\ 
 Univ. Bordeaux, Bordeaux INP, CNRS, IMB, UMR 5251, F-33400, Talence, France \\ 
 Contact: \url{yann.traonmilin@math.u-bordeaux.fr}}

\date{}
\maketitle
\begin{abstract}
	In this note, we give a convergence result for a modified "regularization-by-denoising"(RED) algorithm under a restricted isometry condition on measurements and a restricted Lipschitz  condition on the considered deep projective prior. This study leads to open questions about the convergence of RED algorithms. 
\end{abstract}

\section{Introduction}
 We consider the observation model 
\begin{equation}
 y = A\hat{x} +e
\end{equation}
where $\hat{x} \in \Sigma \subset \bR^N$ and   $A \in \bR^{m \times N} $ and $e \in \bR^m$. The set $\Sigma$ is a low-dimensional model set. It has been shown \cite{traonmilin2024towards,joundi2025stochastic} that under a restricted isometry property (RIP) on $A$ and a restricted Lipschitz condition on a given generalized projection $P_\Sigma$ onto $\Sigma$, the generalized projected gradient descent (GPGD) iterations
\begin{equation}\label{eq:def_GPGD}
	x_{n+1} = P_\Sigma(x_n) - \mu A^T(A P_\Sigma(x_n)-y)
\end{equation}
stably recover $\hat{x}$ with linear rate,  i.e, there is $C> 0$ such that  
\begin{equation}
	\|	x_{n}-\hat{x} \|_2 \leq (\delta\beta)^n\|x_0-\hat{x} \|_2 + C \|e\|_2.
\end{equation}
This result has been extended to stability to model error and approximate restricted Lipschitz projections in \cite{joundi2025sparse}. Such a model has been proposed to describe convergence of plug-and-play (PnP) methods for inverse problems where $P_\Sigma$ is computed by a general purpose denoiser.  This mathematical model has been proven useful for the identification of key geometrical parameters driving identifiability and convergence of PnP. Subsequent regularization of deep projective priors have been proposed to enhance such properties~\cite{joundi2025stochastic,joundi2025sparse}. To the best of our knowledge, the RIP and restricted Lipschitz constants are the weakest known hypotheses leading to such  convergence  guarantees. This theoretical set-up has also been extended to diffusion models for inverse problems~\cite{leong2025recovery}.

Another question is to understand if another popular PnP scheme, regularization by denoising (RED) \cite{romano2017little,reehorst2018regularization,sun2019block,cascarano2024constrained} has convergence guarantees under the same minimal hypotheses.  RED iterations are defined by 

\begin{equation}\label{eq:def_RED}
	x_{n+1} = x_n - \mu A^T(A x_n-y) - \lambda(x_n -P_{\Sigma}(x_n)).
\end{equation}

In this note we show convergence of a slightly modified version of the RED algorithm  under similar (stronger) hypotheses, and we ask  open questions arising from these results.
\section{Notations}
We use the formalism from~\cite{traonmilin2024towards}.

\begin{definition}\label{def:RIC}
	The operator $B$  has restricted isometry constant $\delta \in [0,1)$ on the secant set $\Sigma-\Sigma =\{x_1-x_2 : x_1,x_2 \in \Sigma \}$  if for all $x_1,x_2 \in \Sigma$,
	
	\begin{equation}
		\|(I-B)(x_1-x_2)\|_2\leq \delta \|x_1-x_2\|_2.
	\end{equation}
	We denote by $\delta_\Sigma(B)$ the  smallest restricted isometry constant (RIC) of $B$.
\end{definition}

\begin{definition}[Generalized projection]\label{def:proj}
	Let $\Sigma \subset \bR^N$. A (set-valued) generalized projection onto $\Sigma$ is  a (set-valued) function $P$ such that for any $z\in\bR^N$, $P(z) \subset \Sigma$.
\end{definition}

By abuse of notation, to facilitate reading, an equation true for any $w \in P(z)$ is written using the notation $P(z)$. We introduce orthogonal projections (metric projections for the $\ell^2$ norm) on sets where they exist.

\begin{definition}[Proximinal sets and orthogonal projections]\label{def:orth_proj}
	Let $\Sigma \subset \bR^N$. The set $\Sigma$ is proximinal if for all $z \in \bR^N$, we have 
	\begin{equation}
		\left( \arg \min_{x\in\Sigma} \| x-z \|_2   \right) \neq \emptyset.
	\end{equation}
	
	Now suppose $\Sigma$ is a proximinal set, we define the orthogonal projection onto $\Sigma$ as 
	\begin{equation}
		P_\Sigma^\perp(z)  := \arg \min_{x\in\Sigma} \| x-z \|_2.
	\end{equation}
	Notice that $P_\Sigma^\perp(z)$ may be set-valued.
\end{definition}

\begin{definition}[Restricted Lipschitz property]\label{def:lip_const}
	Consider a generalized projection $P$. Then $P$ has the restricted $\beta$-Lipschitz property with respect to $\Sigma$ iff for all  $z \in \bR^N, x \in \Sigma, u \in P(z)$,  we have
	\begin{equation}
			\|u-x\|_2 \leq \beta \|z-x\|_2.
	\end{equation}
	We denote by $\beta_{\Sigma}(P)$ the smallest $\beta$ such that $P$ has the restricted $\beta$-Lipschitz property.
\end{definition}
\section{Convergence of a modified RED algorithm} \label{sec:conv_mod_RED}

For a parameter $\lambda \in [0,1]$, we remark that RED can be interpreted in a two step algorithm:
\begin{itemize}
	\item Select a point between the current iterate $x_n$ and its projection onto $\Sigma$
	\begin{equation}
		z_n = (1-\lambda)x_n + \lambda P_\Sigma(x_n) \in [x_n, P_\Sigma(x_n)]
	\end{equation}
	 i.e. a convex combination of $x_n$ and $P_\Sigma(x_n)$.
	\item Apply a gradient descent step, with the gradient (of a $\ell^2$ datafit) calculated in $x_n$: 
	\begin{equation}
		x_{n+1} = z_n - \mu A^T(Ax_n-y).
    \end{equation}
\end{itemize}
By rewriting $z_n = x_n - \lambda(x_n-P_\Sigma(x_n))$, we verify that we fall on RED iterations~\eqref{eq:def_RED}.

We propose to look at the following modified RED algorithm, we keep the same definition of $z_n$ but calculate the gradient step in $z_n$ instead of $x_n$: 
\begin{equation}\label{def:mod_RED}
	\begin{split}
	z_n &= (1-\lambda)x_n + \lambda P_\Sigma(x_n) \in [x_n, P_\Sigma(x_n)] \\
	x_{n+1} &= z_n - \mu A^T(Az_n-y).\\
	\end{split}
\end{equation}

We consider the case where we are able to approximate the orthogonal projection onto $\Sigma$ (that has been shown to be at least restricted $2$-Lipschitz) \cite{traonmilin2024towards}.
\begin{theorem}
Suppose $ \Sigma$ is a proximinal set. Consider modified RED iterations~\eqref{def:mod_RED} with $\|P_\Sigma - P_\Sigma^\perp\|_2 \leq \eta$. Let  $ \beta = \beta_\Sigma(P_\Sigma)$ and $\delta = \delta_\Sigma(\mu A^TA)$. Suppose $r = (\delta \beta +	|1-\lambda|\| I - \mu A^TA\|_{\mathrm{op}}) < 1$, then 
	
	\begin{equation}
			\|x_{n}-\hat{x}\|_2 \leq r^n	\|x_0-\hat{x} \|_2  + \frac{1}{1-r}\left(|1-\lambda|\| I - \mu A^TA\|_{\mathrm{op}}\eta + \mu \|A^Te\|_2\right).
	\end{equation}
	
\end{theorem}
\begin{proof}
We have, using the triangle inequality, 

\begin{equation}
	\begin{split}
		&\|x_{n+1}-\hat{x}\|_2 \\&= \|z_n - \mu A^T(Az_n-y)  - \hat{x}\|_2\\	
		&=   \|(I - \mu A^TA)(z_n-\hat{x}) +\mu A^Te\|_2\\
		&\leq 	 \|(I - \mu A^TA)(P_\Sigma(x_n)-\hat{x}))   + (I - \mu A^TA)(z_n-P_\Sigma(x_n))\|_2 + \mu \|A^Te\|_2\\
		&= 	 \|(I - \mu A^TA)(P_\Sigma(x_n)-\hat{x}))   + (I - \mu A^TA)((1-\lambda)x_n +\lambda P_\Sigma(x_n)-P_\Sigma(x_n))\|_2+ \mu \|A^Te\|_2\\
		&= 	\|(I - \mu A^TA)(P_\Sigma(x_n)-\hat{x}))   + (1-\lambda)(I - \mu A^TA)(x_n -P_\Sigma(x_n))\|_2+ \mu \|A^Te\|_2\\
		&\leq 	\|(I - \mu A^TA)(P_\Sigma(x_n)-\hat{x})) \|_2  +	\| (1-\lambda)(I - \mu A^TA)(x_n -P_\Sigma(x_n))\|_2+ \mu \|A^Te\|_2\\
		&\leq 	\|(I - \mu A^TA)(P_\Sigma(x_n)-\hat{x})) \|_2  +	|1-\lambda|\| I - \mu A^TA\|_{\mathrm{op}}\|x_n -P_\Sigma(x_n)\|_2+ \mu \|A^Te\|_2.\\
	\end{split}
\end{equation}
Using the definition of orthogonal projection, we have $\|x_n -P_\Sigma^\perp(x_n)\|_2 \leq  \|x_n -\hat{x}\|_2 $. With the RIC and the restricted Lipschitz property, we deduce

\begin{equation}
	\begin{split}
		&\|x_{n+1}-\hat{x}\|_2\\
		 &\leq \delta \beta 	\|x_n-\hat{x} \|_2  +	|1-\lambda|\| I - \mu A^TA\|_{\mathrm{op}}\|P_\Sigma^\perp(x_n) -\hat{x}\|_2+|1-\lambda|\| I - \mu A^TA\|_{\mathrm{op}}\eta + \mu \|A^Te\|_2\\
			&\leq \delta \beta 	\|x_n-\hat{x} \|_2  +	|1-\lambda|\| I - \mu A^TA\|_{\mathrm{op}}\|x_n -\hat{x}\|_2+|1-\lambda|\| I - \mu A^TA\|_{\mathrm{op}}\eta + \mu \|A^Te\|_2\\
		&= (\delta \beta +	|1-\lambda|\| I - \mu A^TA\|_{\mathrm{op}}) 	\|x_n-\hat{x} \|_2+ |1-\lambda|\| I - \mu A^TA\|_{\mathrm{op}}\eta + \mu \|A^Te\|_2. \\
	\end{split}
\end{equation}
Let $r = \delta \beta +	|1-\lambda|\| I - \mu A^TA\|_{\mathrm{op}}$, by induction, similarly to~\cite{joundi2025stochastic}, we have 

\begin{equation}
	\begin{split}
		\|x_{n}-\hat{x}\|_2 &\leq r^n	\|x_0-\hat{x} \|_2  + \frac{1}{1-r}\left(|1-\lambda|\| I - \mu A^TA\|_{\mathrm{op}}\eta + \mu \|A^Te\|_2\right).\\
	\end{split}
\end{equation}
\end{proof} 

We see that if we chose $\lambda =1$, we fall on GPGD iterations and obtain the same convergence rate  with modified RED (under the additional orthogonality assumption of the projection). Other admissible $\lambda$ lead to slower convergence (and less identifiability).  It leaves the question of convergence to a stable solution with a generic restricted Lipschitz projection without orthogonality assumption (as for GPGD).

\section{On the convergence of classical RED}  
In the context of convex minimization where $x_n-P_\Sigma(x_n)$ can be interpreted as the gradient of a convex function (or $\Sigma$ is a convex set) \cite{cohen2021regularization}, the convergence of such algorithm can be  proved (with typical sub-linear rates). In a non-convex setting, convergence to critical points can be obtained if $x_n-P_\Sigma(x_n)$ is globally Lipschitz (and the gradient of a function) \cite{hurault2022gradient}. To prove convergence in our context with classical proof techniques, we consider $\theta$ such that $ 0< \lambda + \theta \leq 1$ and write iterations~\eqref{eq:def_RED} as follows.
\begin{equation}
	\begin{split}
		x_{n+1}&= (1-\lambda)x_n + \lambda P_\Sigma(x_n)- \mu A^TA(x_n-\hat{x})\\	
		&= (1-\lambda -\theta)x_n + \lambda P_\Sigma(x_n) +\theta x_n -\mu A^TA(x_n-\hat{x})) \\
		&= (1-\lambda -\theta)x_n + (\lambda +\theta) T(x_n)
	\end{split}
\end{equation}
where $T: u  \mapsto \frac{1}{\lambda +\theta} (\lambda P_\Sigma(u)+\theta u- \mu A^TA(u-\hat{x}))$.
The operator $T_{\lambda+\theta}:x_n \mapsto (1-\lambda - \theta)x_n + (\lambda+\theta) T(x_n)$ is called an averaged operator of $T$. The main hypothesis to prove convergence of such fixed point schemes is that  $T$ is quasi-non expensive, i.e. restricted 1-Lipschitz with respect to $\fix (T)$ \cite{bauschkeconvex}. Indeed we verify that $\hat{x} \subset \fix(T)$.   

Applied to our case, we would need the equivalent inequalities
\begin{equation}
	\begin{split}
		\|T(u) - \hat{x}\|_2 &\leq  \|u - \hat{x}\|_2 \\
		L:= \| \frac{1}{\lambda +\theta} (\lambda P_\Sigma(u)+\theta u- \mu A^TA(u-\hat{x}))- \hat{x}\|_2^2 &\leq  \|u - \hat{x}\|_2^2. \\
	\end{split}
	\end{equation} 
We can  write 
	\begin{equation}
		\begin{split}
			L&= \|(I-\frac{\mu}{\lambda+\theta} A^TA)(P_\Sigma(u)-\hat{x}) + \frac{1}{\lambda +\theta} (-\theta P_\Sigma(u)+\theta u- \mu A^TA (u-P_\Sigma(u))\|_2^2 \\
			&= \|(I-\frac{\mu}{\lambda+\theta} A^TA)(P_\Sigma(u)-\hat{x}) + \frac{1}{\lambda +\theta} (\theta I - \mu A^TA )(u-P_\Sigma(u))\|_2^2  \\
			&= \|w_1 +w_2\|_2^2
		\end{split}
	\end{equation} 
	where $w_1 =(I-\frac{\mu}{\lambda+\theta} A^TA)(P_\Sigma(u)-\hat{x})  $ and $w_2=\frac{1}{\lambda +\theta} (\theta I - \mu A^TA )(u-P_\Sigma(u))$.
	
When possible, under suitable scaling of $\frac{\mu}{\lambda +\theta}$,  we have that $\|w_1\|_2 \leq \delta(\frac{\mu}{\lambda+\theta}) \beta \|u -\hat{x}\|_2  < \|u -\hat{x}\|_2$. The main problem is that under the RIP  we cannot generally bound  $w_2$ with a contractive term as $u-P_\Sigma(u)\notin \Sigma-\Sigma$. Moreover, without convexity of the model set $\Sigma$ we cannot really hope to control the  sign of the term $ \ls u-P_\Sigma(u), P_\Sigma(u)-\hat{x} \rs $ if we were to develop the squared norm.

 Suppose we can decompose $u-P_\Sigma(u) = \sum_i \lambda_i v_i$, with $\sum_i \lambda_i =1$ and   $v_i \in \Sigma - \Sigma$ (possible when $\tspan \Sigma = \bR^N$ and $\Sigma$ homogeneous). A way to use the RIC to bound $w_2$ is to use such decomposition. By convexity, we have 
	\begin{equation}
	\begin{split}
		\|w_2\|_2^2& \leq\sum_i \lambda_i \left(\frac{\theta}{\lambda+\theta} \|(I - \frac{\mu}{\theta}A^TA)v_i\|_2 \right)^2\leq \left(\frac{\theta}{\lambda+\theta} \delta(\frac{\mu}{\theta}A^TA)\right)^2 \sum_i \lambda_i\|v_i\|_2^2
	\end{split}
\end{equation} 
taking the infimum over the possible decompositions $(\lambda_i,v_i)_i, v_i \in \Sigma - \Sigma$, we get  
\begin{equation}
	\begin{split}
		\|w_2\|_2& \leq \frac{\theta}{\lambda+\theta} \delta(\frac{\mu}{\theta}A^TA)  \|u-P_\Sigma(u)\|_{\Sigma-\Sigma}
	\end{split}
\end{equation} 
where $\|z\|_{\Sigma-\Sigma} = \inf_{z_i \in \Sigma-\Sigma, \sum_i \lambda_i z_i = z, \sum_i \lambda_i =1} \sqrt{\sum_i\lambda_i\|z_i\|_2^2}$ is the atomic norm induced by $\Sigma-\Sigma$ (see  \cite{traonmilin2018stable}).
Now we cannot generally relate the distance $u-P_\Sigma(u)$ with $u - \hat{x}$ except when $P_\Sigma =P_\Sigma^\perp$ is the orthogonal projection. In this case, we have  $\|u-P_\Sigma(u)\|_2$ and
\begin{equation}
	\begin{split}
		\|w_2\|_2& \leq \frac{\theta}{\lambda+\theta} \delta(\frac{\mu}{\theta}A^TA)  C_\Sigma \|u-\hat{x}\|_2
	\end{split}
\end{equation} 
where $C_\Sigma$ is an equivalence constant between $\|\cdot\|_2$ and $\|\cdot\|_{\Sigma-\Sigma}$. Hence  quasi non-expansiveness would be achieved under the condition 
\begin{equation}
  \delta(\frac{\mu}{\lambda+\theta}A^TA) \beta + \frac{\theta}{\lambda+\theta} \delta(\frac{\mu}{\theta}A^TA)  C_\Sigma \leq 1.
\end{equation} 

To achieve this, we would  need e.g. a small $\theta$, an optimal $\mu$ and $\lambda \approx 1$ and very favorable identifiability conditions of the GPGD scheme  $\delta(\mu A^TA) \beta< 1$ with $\delta(A^TA)\beta$ small.  Maybe the constant $C_\Sigma$ could prevent this from being possible. This leads to following open questions. Is it possible to  show  convergence of the classical RED iterations with RIC and restricted Lipschitz conditions in this theoretical set-up, with less restrictive conditions?

Given the results of the previous Section, what theoretical set-up would warrant the use of a (classical or modified) RED scheme instead of GPGD  as convergence rates and identifiability properties are worse than GPGD? 

This work provides a first analysis of the convergence of a modified version of the RED algorithm under RIC and restricted Lipschitz conditions.
We hope that it  will be the first step into proving the convergence of the original RED algorithm under similar hypotheses.

\bibliographystyle{abbrv}
\bibliography{note_mod_RED.bib}
\end{document}